\documentclass[a4paper,12pt]{article}
\usepackage[utf8]{inputenc}
\usepackage[russian, ukrainian, english]{babel}
\usepackage{textcomp}
\usepackage{amsmath, amsthm, amsfonts, amssymb}
\usepackage{marvosym}
\newcommand{\E}{{\mathbb{E}}}
\newcommand{\R}{{\mathbb{R}}}

\newcommand{\ov}{\overline}

\newcommand{\vf}{\varphi}
\newcommand{\cT}{{\mathcal T}}
\newcommand{\cG}{{\mathcal G}}
\newcommand{\cP}{{\mathcal P}}
\newcommand{\cS}{{\mathcal S}}
\newcommand{\cE}{{\mathcal E}}
\newcommand{\wt}{\widetilde}
\newcommand{\1}{\mathbf{1}}
\newcommand{\mbR}{{\mathbb R}}

\newcommand{\cF}{{\mathcal F}}

\newcommand{\mbN}{{\mathbb N}}

\newcommand{\mbZ}{{\mathbb Z}}

\newcommand{\pt}{\partial}
\newcommand{\ve}{\varepsilon}
\newcommand{\wh}{\widehat}

\renewcommand{\ker}{\mathop{\rm ker}}

\theoremstyle{plain}
\newtheorem{thm}{Theorem}
\newtheorem{lem}{Lemma}

\newtheorem{corollary}{Corollary}
\theoremstyle{definition}
\newtheorem{defn}{Definition}

\begin{document}
UDC 519.21
\begin{center}
{\Large\bf Clark representation for self-intersection local times of Gaussian integrators}\\[1cm]
A.A. Dorogovtsev, O.L. Izyumtseva, N. Salhi
\end{center}
andrey.dorogovtsev@gmail.com, olgaizyumtseva@gmail.com, salhi.naoufel@gmail.com\\
Institute of Mathematics National Academy of Sciences of Ukraine, 01004 Ukraine, Kiev-4,
3, Tereschenkivska st.;
University of Tunis El Manar, Rommana 1068, Tunisie\\
\vskip20pt

In present article we prove the existence of multiple self-intersection local times, describe its It\^{o}-Wiener expansion and establish Clark representation for the class of Gaussian integrators generated by operators with a finite dimensional kernel.
\section{Introduction}
Let $w(t),\ t\in[0, 1]$ be one-dimensional Wiener process. The following well-known Clark theorem describes functionals from $w$ in terms of It\^{o} stochastic integral.

\begin{thm}(Clark) \cite{1}
\label{thm1}
If a square-integrable random variable $\alpha$ is measurable with respect to $w,$ then it can be uniquely represented as
$$
\alpha=E\alpha +\int^1_0\eta(s)dw(s).
$$
\end{thm}
Here the process $\eta$ is adapted to the filtration generated by $w$ and satisfies relation
$$
E\int^1_0\eta(s)^2ds=Var\ \alpha.
$$
The farther development of this statement goes in two directions. The first one is to try to find precise form of $\eta$ for different special functionals $\alpha$ from $w.$  As an example of such special functional one can take a local time of Wiener process at the point $u\in{\mathbb R}$ \cite{2} which is formally defined by the formula
$$
\int^1_0\delta_u(w(s))ds.
$$
It was mentioned in \cite{3} that Clark representation for this local time has a form
$$
\int^1_0\delta_u(w(s))ds=\int^1_0p_r(u)dr+
$$
$$
+\int^1_0\Big(\int^1_rp'_{s-r}(w(r)-u)ds\Big)dw(r).
$$
Here $p_t$ is a density of normal distribution with mean zero and variance $t.$ If $\alpha$ has a stochastic derivative \cite{4}, then the process $\eta$ can be expressed through Clark-Ocone formula \cite{16}. But there are interesting cases when $\alpha $ is not stochastically differentiable.

The second direction of development of Clark formula is to establish it for more general processes than Wiener process.

This article is devoted to both mentioned problems. We will consider Gaussian integrators \cite{5} instead of Wiener process. These processes have properties which allow to construct Skorokhod integral \cite{4} with respect to them. So one can ask about Clark representation with Skorokhod integral. In the paper \cite{3} the version of Clark  representation with Skorokhod integral was obtained for the local time of integrators. In the present article we consider Clark representation for the self-intersection local time of one dimensional Gaussian integrators.

The paper is organized as follows:

In Section 2 we prove the existence of self-intersection local time for one-dimensional Gaussian integrator generated by a continuous linear operator with the finite-dimensional kernel.

In Section 3 we present It\^{o}-Wiener expansion for the self-intersection local time of Gaussian integrators.

Section 4 is devoted to Clark representation for self-intersection local time of Gaussian integrators.
\section{Existence of self-intersection local time for Gaussian integrators}
In present section we discuss the existence of self-intersection local time for one-dimensional Gaussian integrators. Every Gaussian integrator can be associated with some continuous linear operator in the space $L_2([0;1])$ via white noise representation. Since that properties of functionals from Gaussian integrator can be discussed in terms of properties of corresponding linear operator. To do this we need a notion of white noise in separable Hilbert space $H$.
\begin{defn}\cite{4}
\label{defn1}
A linear correspondence $\xi$ such that
$$
H\ni h\longrightarrow (h,\xi)\sim N(0,\|h\|^2)
$$
is said to be a white noise in Hilbert space $H.$
\end{defn}
Important example of a white noise is the following. Let
$H=L_2([0;1])$ and $w(t),\ t\in[0;1]$ be one-dimensional Winer process. Put for any $h\in L_2([0;1])$
$$
(h,\xi)=\int^{1}_{0}h(t)dw(t).
$$
It follows from the properties of stochastic integral that $\xi$ is a white noise in the space $L_2([0;1]).$
\begin{defn} (A.A. Dorogovtsev, 1998) \cite{5}
\label{defn2}
A centered Gaussian process $x(t),\ t\in[0; 1],\ x(0)=0$ is said to be an integrator if there exists a constant $c>0$ such that for an arbitrary partition $0=t_0<t_1<\ldots<t_n=1$ and real numbers $a_0, \ldots, a_{n-1} $ the following inequality holds
\end{defn}
\begin {equation}
\label{eq21}
 E\Big(\sum^{n-1}_{k=0}a_k(x(t_{k+1})-x(t_k))\Big)^2\leq c\sum^{n-1}_{k=0}a^2_k\Delta t_k.
\end{equation}

Relation \eqref{eq21} allows for any square integrable function on the interval $[0;1]$ to define a stochastic integral with respect to a Gaussian integrator. Indeed, the sum
$$
\sum^{\infty}_{k=0}a_k(x(t_{k+1})-x(t_k))
$$
can be considered as the integral
$$
\int^1_0f(t)dx(t)
$$
of the step function
$$
f=\sum^{n-1}_{k=0}a_k\ \1_{[t_k;t_{k+1})}
$$
over the process $x$ and \eqref{eq21} means that
$$
\Big\|\int^{1}_{0}f(t)dx(t)\Big\|^2_{L_2(\Omega,\mathcal{F},P)}\leq c\|f\|^2_{L_2([0;1])}.
$$
It implies that the mapping
$$
f\ \mapsto\ \int^1_0f(t)dx(t)
$$
can be extended to a continuous linear operator on the entire space $L_2([0;1]).$ The result of this extension is said to be the integral over the process $x.$

The following lemma describes a structure of Gaussian integrators.
\begin{lem} \cite{5}
\label{lem8}
A centered Gaussian process $x(t),\ t\in[0; 1]$ is an integrator iff there exist a white noise $\xi$ in $L_2([0; 1])$ and a continuous linear operator $A$ in $L_2([0; 1])$ such that
\begin{equation}
\label{eq22}
x(t)=(A\1_{[0; t]}, \xi),\ t\in[0; 1].
\end{equation}
\end{lem}
It follows from the representation \eqref{eq22} that all properties of Gaussian integrator can be described in terms of operator $A.$ Examples of integrators are Wiener process, Brownian bridge, Fractional Brownian motion with Hurst parameter $\alpha\geq\frac{1}{2}$ \cite{19}.

The main aim of present section is to find conditions on the operator $A$ in the representation \eqref{eq22} that guarantee the existence of self-intersection local time.

Let us recall the definition of self-intersection local time for a one dimensional random process $x(t),\ t\in[0; 1].$ For $0\leq a<b$
denote by
$$
\Delta_k(a,b)=\{a\leq t_1\leq\ldots\leq t_k\leq b\}.
$$

Put $\Delta_k(0,1)=:\Delta_k.$
Consider the family of approximations
$$
T^{x}_{\ve,k}=\int_{\Delta_k}\prod^{k-1}_{i=1}p_\ve(x(t_{i+1})-x(t_i))d\vec{t}.
$$
As above
$$
p_\ve(z)=\frac{1}{\sqrt{2\pi\ve}}e^{-\frac{z^2}{2\ve}}.
$$

\begin{defn}
\label{defn3}
For $p\in \mbN$ the random variable
$$
T^{x}_k=L_p\mbox{-}\lim_{\ve\to0}T^{x}_{\ve,k}
$$
is said to be the self-intersection local time of multiplicity $k$ for the process $x$ whenever the limit exists.
\end{defn}
Consider a Gaussian integrator
\begin{equation}
\label{eq1}
x(t)=(A\1_{[0;t]}, \xi),\ t\in [0;1],
\end{equation}
where $A$ is a continuous linear operator in the space $L_2([0;1])$ and $\xi $ is a white noise in the same space. Let $\tilde{A}$ be a restriction of operator $A$ onto $(\ker A)^\perp.$ The main statement of present section is the following theorem about the existence of self-intersection local time for a one dimensional Gaussian integrator.

\begin{thm}
\label{thm2}
Suppose that a continuous linear operator $A$ in the space $L_2([0; 1])$ satisfies conditions

1) $\dim\ker A<+\infty;$

2) the operator  $\tilde{A}$ is continuously invertible.

Then for the Gaussian integrator \eqref{eq1} generated by the operator $A$ for any $p\in \mbN,\ k\geq 2$ there exists the self-intersection local time
$$
T^{x}_k=L_p\mbox{-}\lim_{\ve\to0}T^{x}_{\ve,k}.
$$
\end{thm}
\begin{proof}
Let us check that $\E \;(T^{x}_{\ve_1,k}-T^{x}_{\ve_2,k})^{2p}\to 0,\ \ve\to 0.$ Note that
$$
\E \; (T^{x}_{\ve_1,k}-T^{x}_{\ve_2,k})^{2p}=
$$
$$
=\sum^{2p}_{l=0}(-1)^{2p-l}C^{l}_{2p}\; \E \;(T^{x}_{\ve_1,k})^{l}(T^{x}_{\ve_{2,k}})^{2p-l}.
$$
Therefore, to prove the theorem it suffices to check that there exists the finite limit
$\E \; (T^{x}_{\ve_1,k})^{l}(T^{x}_{\ve_2,k})^{2p-l}$ as $\ve\to 0.$ In the beginning consider the case $\ve_1=\ve_2$ and check that there exists the finite limit
$\E \; (T^{x}_{\ve,k})^{2p}$ as $\ve\to 0.$ Let $G(e_1, \ldots, e_n)$ be Gram determinant constructed from elements $e_1,\ldots, e_n$ and $B(e_1, \ldots, e_n)$ be  corresponding Gramian matrix. Note that $B(A1_{[t^1_1;t^1_2]},\ldots,A1_{[t^{2p}_{k-1};t^{2p}_{k}]})$ is a covariance matrix of Gaussian vector $\vec{x}(\vec{t}).$ Put $\Delta^{2p}_{k}=\Delta_k\times\ldots\times\Delta_k,$ where the product contains $2p$ terms.
Notice that
$$
\E \;(T^{x}_{\ve,k})^{2p}=\; \E \; \Big(\int_{\Delta_k}\prod^{k-1}_{i=1}p_\ve(x(t_{i+1})-x(t_i)d\vec{t}\Big)^{2p}=
$$
\begin{equation}
\label{eq170}
=\int_{\Delta^{2p}_{k}}E\prod^{k-1}_{i=1}p_\ve(x(t^1_{i+1})-x(t^1_i))\ldots \prod^{k-1}_{i=1}p_\ve(x(t^{2p}_{i+1})-x(t^{2p}_i))d\vec{t}.
\end{equation}
Denote by
$$
\vec{x}(\vec{t})=(x(t^1_2)-x(t^1_1),\ldots,x(t^{2p}_{k})-x(t^{2p}_{k-1})),
$$
$$
p^{2p(k-1)}_{\ve}(y)=\frac{1}{(2\pi\ve)^{p(k-1)}}e^{-\frac{\|y\|^2}{2\ve}},\ y\in\mbR^{2p(k-1)}.
$$
Then \eqref{eq170}  equals
$$
\int_{\Delta^{2p}_{k}}Ep^{2p(k-1)}_{\ve}(\vec{x}(\vec{t}))d\vec{t}=
$$
$$
=\int_{\Delta^{2p}_{k}}\int_{\mbR^{2p(k-1)}}\frac{1}{(2\pi\ve)^{p(k-1)}}e^{-\frac{1}{2}(\frac{1}{\ve}Iy,y)}
\frac{1}{(2\pi)^{p(k-1)}\sqrt{G(A\1_{[t^1_1;t^1_2]},\ldots,A\1_{[t^{2p}_{k-1};t^{2p}_{k}]})}}\cdot
$$
$$
\cdot e^{-\frac{1}{2}(B^{-1}(A\1_{[t^1_1;t^1_2]},\ldots,A\1_{[t^{2p}_{k-1};t^{2p}_{k}]})y,y)}dy d\vec{t}=
$$
$$
=\int_{\Delta^{2p}_{k}}\int_{\mbR^{2p(k-1)}}\frac{1}{(2\pi\ve)^{p(k-1)}}\frac{1}{(2\pi)^{p(k-1)}\sqrt{G(A\1_{[t^1_1;t^1_2]},\ldots,A\1_{[t^{2p}_{k-1};t^{2p}_{k}]})}}\cdot
$$
$$
\cdot e^{-\frac{1}{2}((\frac{1}{\ve}I+B^{-1}(A\1_{[t^1_1;t^1_2]},\ldots,A\1_{[t^{2p}_{k-1};t^{2p}_{k}]}))y,y)}dy d\vec{t}=
$$
$$
=\int_{\Delta^{2p}_{k}}\int_{\mbR^{2p(k-1)}}\frac{1}{(2\pi)^{p(k-1)}\ve^{p(k-1)}\sqrt{G(A\1_{[t^1_1;t^1_2]},\ldots,A\1_{[t^{2p}_{k-1};t^{2p}_{k}]})}}\cdot
$$
$$
\cdot\frac{1}{\sqrt{\det(\frac{1}{\ve}I+B^{-1}(A\1_{[t^1_1;t^1_2]},\ldots,A\1_{[t^{2p}_{k-1};t^{2p}_{k}]}))}}\cdot
$$
$$
\cdot\frac{1}{(2\pi)^{p(k-1)}\sqrt{\det(\frac{1}{\ve}I+B^{-1}(A\1_{[t^1_1;t^1_2]},\ldots,A\1_{[t^{2p}_{k-1};t^{2p}_{k}]}))^{-1}}}\cdot
$$
$$
\cdot e^{-\frac{1}{2}((\frac{1}{\ve}I+B^{-1}(A\1_{[t^1_1;t^1_2]},\ldots,A\1_{[t^{2p}_{k-1};t^{2p}_{k}]}))y,y)}dy d\vec{t}=
$$
$$
=\int_{\Delta^{2p}_{k}}\frac{1}{(2\pi)^{p(k-1)}\ve^{p(k-1)}\sqrt{G(A\1_{[t^1_1;t^1_2]},\ldots,A\1_{[t^{2p}_{k-1};t^{2p}_{k}]})}}\cdot
$$
$$
\cdot\frac{1}{\sqrt{\det(\frac{1}{\ve}I+B^{-1}(A1_{[t^1_1;t^1_2]},\ldots,A1_{[t^{2p}_{k-1};t^{2p}_{k}]}))}}d\vec{t}.
$$
One can check that
$$
=\int_{\Delta^{2p}_{k}}\frac{1}{(2\pi)^{p(k-1)}\ve^{p(k-1)}\sqrt{G(A\1_{[t^1_1;t^1_2]},\ldots,A\1_{[t^{2p}_{k-1};t^{2p}_{k}]})}}\cdot
$$
$$
\cdot\frac{1}{\sqrt{\det(\frac{1}{\ve}I+B^{-1}(A\1_{[t^1_1;t^1_2]},\ldots,A\1_{[t^{2p}_{k-1};t^{2p}_{k}]}))}}d\vec{t}=
$$
$$
=\int_{\Delta^{2p}_{k}}\frac{1}{(2\pi)^{p(k-1)}}\cdot
$$
$$
\cdot\frac{1}{\sqrt{\ve^{2p(k-1)}+\ve^{2p(k-1)-1}\sum^{2p}_{i=1}\|A\1_{[t^{i}_{k-1};t^{i}_{k}]}\|^2+\ldots+G(A\1_{[t^1_1;t^1_2]},\ldots,A\1_{[t^{2p}_{k-1};t^{2p}_{k}]})}}
d\vec{t}.
$$
Indeed, to make the calculations easy let us prove it for $p=1,\ k=2.$ Notice that
$$
\det\Big(\frac{1}{\ve}I+B^{-1}(A\1_{[t^1_1;t^1_2]},A\1_{[t^{2}_{1};t^{2}_{2}]})\Big)=
$$
$$
=\frac{1}{G(A\1_{[t^1_1;t^1_2]},A\1_{[t^{2}_{1};t^{2}_{2}]})}+\frac{1}{\ve}\sum^{2}_{i=1}\frac{\|A\1_{[t^i_1;t^i_2]}\|^2}{G(A\1_{[t^1_1;t^1_2]},A\1_{[t^{2}_{1};t^{2}_{2}]})}
+\frac{1}{\ve^2}
$$
Consequently
$$
\frac{1}{2\pi\ve\sqrt{G(A\1_{[t^1_1;t^1_2]},A\1_{[t^{2}_{1};t^{2}_{2}]})}\sqrt{\det(\frac{1}{\ve}I+B^{-1}(A\1_{[t^1_1;t^1_2]},A\1_{[t^{2}_{1};t^{2}_{2}]}))}}=
$$
$$
=\frac{1}{2\pi\sqrt{\ve^2+\ve\sum^{2}_{i=1}\|A\1_{[t^{i}_{1};t^{i}_{2}]}\|^2+G(A\1_{[t^1_1;t^1_2]},A\1_{[t^{2}_{1};t^{2}_{2}]})}}.
$$
Hence to apply Lebesgue's dominated convergence theorem one can check that the integral
\begin{equation}
\label{eq2}
\int_{\Delta^{2p}_{k}}\frac{1}{\sqrt{G \big( A\1_{[t^{1}_1;t^{1}_2]}, \ldots,A\1_{[t^{2p}_{k-1};t^{2p}_{k}]}\big)}}d\vec{t}
\end{equation}
converges.
Denote by $P$ a projection onto $\ker A.$ Then \eqref{eq2}  equals
\begin{equation}
\label{eq3}
\int_{\Delta^{2p}_{k}}\frac{1}{\sqrt{G(\tilde{A}(I-P)\1_{[t^{1}_1;t^{1}_2]}, \ldots,\tilde{A}(I-P)\1_{[t^{2p}_{k-1};t^{2p}_{k}]})}}d\vec{t}.
\end{equation}
To check that the integral \eqref{eq3} converges one can use the following statement.

\begin{thm}\cite{8}
\label{thm3}
Suppose that $A$ is a continuously invertible operator in Hilbert space $H.$ Then for any elements $e_1, \ldots, e_n$ of  space $H$ the following relation holds
$$
G(Ae_1, \ldots, Ae_n)\geq\frac{1}{\|A^{-1}\|^{2n}}G(e_1, \ldots, e_n).
$$
\end{thm}
It follows from Theorem \ref{thm3} that \eqref{eq3} less or equal to
\begin{equation}
\label{eq4}
\frac{1}{\|\tilde{A}^{-1}\|^{2p}}\int_{\Delta^{2p}_{k}}\frac{1}{\sqrt{G((I-P)\1_{[t^{1}_1;t^{1}_2]}, \ldots,(I-P)\1_{[t^{2p}_{k-1};t^{2p}_{k}]})}}d\vec{t}.
\end{equation}

To estimate the integrand in \eqref{eq4} we need

\begin{lem} \cite{10}
\label{lem1}
Let $L$ be a finite-dimensional subspace of space $L_2([0; 1]),\ P_L$ be a projection on $L.$ Suppose that $e_1,\ldots, e_m$ is an orthonormal basis in $L.$  Then for any $g_1,\ldots, g_k\in L_2([0; 1])$ the following relation holds
$$
G((I-P_L)g_1, \ldots, (I-P_L)g_k)=G(g_1, \ldots, g_k, e_1, \ldots, e_m).
$$
\end{lem}
The statement of Lemma \ref{lem1} is a generalization of Cavalieri's principle.

Let $q_1,\ldots, q_m$ be an orthonormal basis in $\ker A.$ Then it follows from Lemma \ref{lem1}  that the Gram determinant in \eqref{eq4} possesses representation
$$
G((I-P)\1_{[t^1_1; t^1_2]}, \ldots,(I-P)\1_{[t^{2p}_{k-1}; t^{2p}_k]} )=
G(\1_{[t^1_1; t^1_2]}, \ldots,\1_{[t^{2p}_{k-1}; t^{2p}_k]}, q_1, \ldots, q_m).
$$
Let us describe the set
$$
\{\vec{t}\in\Delta_{k}^{2p}: \ G(\1_{[t^1_1; t^1_2]}, \ldots,\1_{[t^{2p}_{k-1}; t^{2p}_k]}, q_1, \ldots, q_m)=0\}.
$$
Note that
$$
G(\1_{[t^1_1; t^1_2]}, \ldots,\1_{[t^{2p}_{k-1}; t^{2p}_k]}, q_1, \ldots, q_m)=0
$$
iff there exist $\alpha_1, \ldots, \alpha_{k-1}$ such that $\alpha^2_1+\ldots+\alpha^2_{k-1}>0$ and $\beta_1, \ldots, \beta_m$ which satisfy relation
\begin{equation}
\label{eq7}
\sum^{2p}_{j=1}\sum^{k-1}_{i=1}\alpha^{j}_i\1_{[t^{j}_{i}; t^{j}_{i+1}]}=\sum^m_{j=1}\beta_jq_j.
\end{equation}

Hence if $G(\1_{[t^1_1; t^1_1]}, \ldots,\1_{[t^{2p}_{k-1}; t^{2p}_k]}, q_1, \ldots, q_m)=0,$ then step functions belong to $\ker A.$ Let $L$ be a subspace generated by step functions in $\ker A$ and $\{f_k,\ k=\ov{1,s}\}$ be an orthonormal basis in $L.$ Suppose that  $e_1, \ldots, e_n$ is an orthonormal basis in an orthogonal complement of $L$ in $\ker A.$ Note that $f_1, \ldots, f_s, e_1, \ldots, e_n$ is an orthonormal basis in $\ker A$ and for any $\beta_1, \ldots, \beta_n$
$$
\sum^n_{j=1}\beta_je_j\perp L.
$$
Let us check that
\begin{equation}
\label{eq8}
\int_{\Delta_k^{2p}}\frac{d\vec{t}}
{\sqrt{G(\1_{[t^1_1; t^1_2]}, \ldots,\1_{[t^{2p}_{k-1}; t^{2p}_k]}, f_1, \ldots, f_s, e_1, \ldots, e_n )}}<+\infty.
\end{equation}
To prove \eqref{eq8} we need the following statements.
\begin{lem}\cite{10}
\label{lem3}
There exists a positive constant $c$ which depends on $f_1,\ldots,f_s$ and $e_1,\ldots, e_n$ such that for any $t_1,\ldots, t_k\in\Delta_k$
$$G(\1_{[t_1; t_2]}, \ldots,\1_{[t_{k-1}; t_k]}, f_1, \ldots, f_s, e_1, \ldots, e_n )\geq
c\cdot G(\1_{[t_1; t_2]}, \ldots,\1_{[t_{k-1}; t_k]}, f_1, \ldots, f_s).
$$
\end{lem}

\begin{lem} \cite{10}
\label{lem4}
Let $0< s_1<\ldots<s_N<1$ be the points of jumps of functions $f_1, \ldots, f_s.$ Then there exists a positive constant $c_{\vec{s}}$ which depends on $\vec{s}=(s_1, \ldots, s_N)$ such that
$$
G(\1_{[t_1;t_2]}, \ldots, \1_{[t_{k-1};t_k]},f_1, \ldots, f_s)\geq
$$
\begin{equation}
\label{eq9}
\geq c_{\vec{s}}\ G(\1_{[t_1;t_2]}, \ldots, \1_{[t_{k-1};t_k]},\1_{[0;s_1]}, \1_{[s_1;s_2]}, \ldots, \1_{[s_{N-1};s_N]},\1_{[s_N;1]}).
\end{equation}
\end{lem}
It follows from Lemma \ref{lem3}, Lemma \ref{lem4} that to finish the proof of Theorem  \ref{thm2}   one can check that
$$
\int_{\Delta^{2p}_{k}}
\frac{d\vec{t}}
{\sqrt{G(\1_{[t^1_1; t^1_2]}, \ldots,\1_{[t^{2p}_{k-1}; t^{2p}_k]},\1_{[0; s_1]}, \1_{[s_1; s_2]}, \ldots, \1_{[s_{N-1}; s_N]},\1_{[s_N;1]})}}<+\infty.
$$
Note that
$$
\int_{\Delta^{2p}_{k}}
\frac{d\vec{t}}
{\sqrt{G(\1_{[t^1_1; t^1_2]}, \ldots,\1_{[t^{2p}_{k-1}; t^{2p}_k]},\1_{[0; s_1]}, \1_{[s_1; s_2]}, \ldots, \1_{[s_{N-1}; s_N]},\1_{[s_N;1]})}}=
$$
$$
=\tilde{c}_{\vec{s}}\int_{\Delta^{2p}_{k}}
\frac{d\vec{t}}
{\sqrt{G(\1_{[t^1_1; t^1_2]}, \ldots,\1_{[t^{2p}_{k-1}; t^{2p}_k]},\frac{\1_{[0; s_1]}}{\sqrt{s_1}},\frac{ \1_{[s_1; s_2]}}{\sqrt{s_2-s_1}}, \ldots, \frac{\1_{[s_{N-1}; s_N]}}{\sqrt{s_N-s_{N-1}}},\frac{\1_{[s_N;1]}}{\sqrt{1-s_N}})}}=
$$
\begin{equation}
\label{eq10}
=\tilde{c}_{\vec{s}}\int_{\Delta^{2p}_{k}}
\frac{d\vec{t}}
{\sqrt{G((I-\tilde{P})\1_{[t^1_1; t^1_2]}, \ldots,(I-\tilde{P})\1_{[t^{2p}_{k-1}; t^{2p}_k]})}},
\end{equation}

where $\tilde{P}$ is the projection onto the linear span generated by $\{\1_{[0; s_1]},\ldots,\1_{[s_N;1]}\}.$
If integral \eqref{eq10} converges, then one can give it the following meaning
$$
\int_{\Delta^{2p}_{k}}
\frac{d\vec{t}}
{\sqrt{G((I-\tilde{P})\1_{[t^1_1; t^1_2]}, \ldots,(I-\tilde{P})\1_{[t^{2p}_{k-1}; t^{2p}_k]})}}=
$$
$$
=(2\pi)^{k-1}\lim_{\ve\to0}\ \E \; \Big(\int_{\Delta_k}\prod^{k}_{i=1}p_{\ve}(y(t_{i+1})-y(t_i))d\vec{t}\Big)^{2p}=
$$
$$
=(2\pi)^{k-1}\ \E \; \Big(\int_{\Delta_k}\prod^{k}_{i=1}\delta_0(y(t_{i+1})-y(t_i))d\vec{t}\Big)^{2p},
$$
where
$$
y(t)=((I-\tilde{P})\1_{[0;t]},\xi),\ t\in[0;1]=^{d}
$$
\begin{equation}
\label{eq11}
=^{d}
\begin{cases}
w_1(t)-\frac{t}{s_1}w_1(s_1),&t\in[0;s_1]\\
w_2(t-s_1)-\frac{t-s_1}{s_2-s_1}w_2(s_2-s_1),&t\in[s_1;s_2]\\
\ldots\\
w_{N+1}(t-s_N)-\frac{t-s_N}{1-s_N}w_{N+1}(1-s_N),&t\in[s_N;1]
\end{cases}
\end{equation}
Here $w_1,\ldots,w_N$ are independent Wiener processes. Using the representation \eqref{eq11} for \eqref{eq10}, one can see that it suffices to check that there exists
$$
L_p\mbox{-}\lim_{\ve\to0}\int_{\Delta_k}\prod^{k-1}_{i=1}p_{\ve}(\eta(t_{i+1})-\eta(t_i))d\vec{t}=
$$
$$
=:\int_{\Delta_k}\prod^{k-1}_{i=1}\delta_0(\eta(t_{i+1})-\eta(t_i))d\vec{t},
$$
where $\eta(t)=w(t)-tw(1),\ t\in [0;1]$ is Brownian bridge. Integral \eqref{eq10} in the case of Brownian bridge has the following representation
\begin{equation}
\label{eq400}
\int_{\Delta^{2p}_{k}}
\frac{d\vec{t}}
{\sqrt{G((I-P_{\1_{[0; 1]}})\1_{[t^1_1; t^1_2]}, \ldots,(I-P_{\1_{[0; 1]}})\1_{[t^{2p}_{k-1}; t^{2p}_k]})}},
\end{equation}
where $P_{\1_{[0; 1]}}$ is the projection onto the linear subspace generated by $\1_{[0; 1]}.$
It follows from Lemma \ref{lem1} that \eqref{eq400} equals
\begin{equation}
\label{eq401}
\int_{\Delta^{2p}_{k}}
\frac{d\vec{t}}
{\sqrt{G(\1_{[t^1_1; t^1_2]}, \ldots,\1_{[t^{2p}_{k-1}; t^{2p}_k]},\1_{[0; 1]})}}.
\end{equation}
Let us check that \eqref{eq401} converges. To do this we need
\begin{lem} \cite{9}
\label{lem5}
Let $\Delta_0=\O,$ and $\Delta_1, \ldots, \Delta_n$ be subsets of $[0; 1].$ Then
$$
G(\1_{\Delta_1}, \ldots, \1_{\Delta_n})\geq\prod^n_{k=1}|\Delta_k\setminus\mathop{\cup}\limits^{k-1}_{j=1}\Delta_j|.
$$
\end{lem}

As a consequence of Lemma \ref{lem5}  one can obtain the following estimate for Gram determinant

$$
G(\1_{[t^1_1; t^1_2]}, \ldots,\1_{[t^{2p}_{k-1}; t^{2p}_k]},\1_{[0; 1]})\geq\prod^N_{j=1}|\wt{\Delta}_j|,
$$
where $\wt{\Delta}_j,\ j=1, \ldots, N$ are intervals from the partition of $[0; 1]$ by end-points of intervals $[t^1_1, t^1_2],\ldots, [t^{2p}_{k-1}, t^{2p}_k].$
Hence
$$
\int_{\Delta^{2p}_{k}}
\frac{d\vec{t}}
{\sqrt{G(\1_{[t^1_1; t^1_2]}, \ldots,\1_{[t^{2p}_{k-1}; t^{2p}_k]},\1_{[0; 1]})}}\leq
$$
$$
\int_{\Delta_{2pk}}
\frac
{1}
{\sqrt{\prod^{2pk}_{i=0}(t_{i+1}-t_i)}}d\vec{t},
$$
where $t_0=0$ and $t_{2pk+1}=1.$ Integrating over $t_1$ and $t_{2pk} $ one can obtain that
$$
\int_{\Delta_{2pk}}
\frac
{1}
{\sqrt{\prod^{2pk}_{i=0}(t_{i+1}-t_i))}}d\vec{t}\leq
$$
$$
\leq c\int_{\Delta_{2pk-2}}
\frac
{1}
{\sqrt{\prod^{2pk-3}_{i=1}(t_{i+1}-t_i)}}d\vec{t},\ c>0.
$$
To calculate $\int_{\Delta_{2pk-2}}
\frac
{1}
{\sqrt{\prod^{2pk-3}_{i=1}(t_{i+1}-t_i)}}d\vec{t}$ we need
\begin{lem}
\label{lem6}
For any integer $k\geq 1$ and real numbers $0 \; \leqslant \; a \; < \; b  $
the following identity holds
\begin{equation}
\label{eq200}
\int_{\Delta _k (a,b)} \dfrac{d\vec{t}}{\sqrt{\prod _{i=1} ^{k-1}(t_{i+1}-t_i )}  } = \dfrac{\pi ^{\frac{k-1}{2}} (b-a)^{\frac{k+1}{2}}}{\Gamma (\frac{k+3}{2})} .
\end{equation}
\end{lem}
\begin{proof}
Denote by
$$  H_k(t)= \int_{\Delta _{k}(a,t) } \dfrac{dt_1\ldots dt_k}{\sqrt{(t-t_k)\prod _{i=1} ^{k-1}(t_{i+1}-t_i )}  } \; ,   $$
$$  I _k(t) = \int_{\Delta _{k}(a,t) } \dfrac{dt_1\ldots dt_k}{\sqrt{\prod _{i=1} ^{k-1}(t_{i+1}-t_i )}  } \; ,\ t>a.   $$
One can see that
\begin{equation}
\label{eq17}
 I _{k+1}(t)=\int_a^t H_{k}(u)du
\end{equation}
and
\begin{equation}
\label{eq18}
H_{k+1}(t)= \int_a^t \dfrac{H_{k}(u)}{\sqrt{t-u}}du = \sqrt{t-a} \int_0^1 \dfrac{H_{k}(a+(t-a)\theta)}{\sqrt{1-\theta }}d\theta \; .
\end{equation}
Note that
$$ \displaystyle H_2(t)=2(t-a) B \left( \dfrac{3}{2},\dfrac{1}{2} \right).  $$
It follows from \eqref{eq18} that
$$
H_k(t) = (t-a)^{\frac{k}{2}}\prod _{r=2}^{k+1}B \left( \dfrac{r}{2},\dfrac{1}{2} \right)  = \dfrac{(t-a)^{\frac{k}{2}} \pi ^{\frac{k}{2}} }{\Gamma (\frac{k+2}{2})} \; .
$$
Using \eqref{eq17} and integrating the function $H_{k-1}$ from $a$ to $b$ one can get \eqref{eq200}.
\end{proof}
It follows from Lemma \ref{lem6} that
$$
\int_{\Delta_{2pk-2}}
\frac
{1}
{\sqrt{\prod^{2pk-3}_{i=1}(t_{i+1}-t_i)}}d\vec{t}=\dfrac{\pi ^{\frac{2pk-3}{2}}} {\Gamma (\frac{2p+1}{2})}.
$$
 In the case $\ve_1\neq\ve_2$ notice that for any $y\in\mbR$
$$
p_{\ve_1}(y)=\sqrt{\frac{\ve_2}{\ve_1}}p_{\ve_2}\Big(\sqrt{\frac{\ve_2}{\ve_1}}y\Big).
$$
Then
$$
ET^{x}_{\ve_1,2}T^{x}_{\ve_2,2}=
$$
$$
=\int_{\Delta^2_2}\sqrt{\frac{\ve_2}{\ve_1}}E p_{\ve_2}\Big(\sqrt{\frac{\ve_2}{\ve_1}}(x(t^1_2)-x(t^1_1))\Big)p_{\ve_2}(x(t^2_2)-x(t^2_1))d\vec{t}=
$$
$$
=\int_{\Delta^2_2}\frac{1}{2\pi\sqrt{\ve_1\ve_2+\ve_1\|A\1_{[t^{1}_{1};t^{1}_{2}]}\|^2+\ve_2\|A\1_{[t^{2}_{1};t^{2}_{2}]}\|^2+G(A\1_{[t^1_1;t^1_2]},A\1_{[t^{2}_{1};t^{2}_{2}]})}}d\vec{t}.
$$
Similarly, for $l=\overline{1,2p-1}$
$$
E(T^{x}_{\ve_1,k})^lT^{x}_{\ve_2,k})^{2p-l}=
$$
$$
=\int_{\Delta^{2p}_{k}}\frac{1}{(2\pi)^{p(k-1)}}\cdot
$$
$$
\cdot\frac{1}{\sqrt{\ve_1^{l(k-1)}\ve_2^{(2p-l)(k-1)}+\ldots+G(A\1_{[t^1_1;t^1_2]},\ldots,A\1_{[t^{2p}_{k-1};t^{2p}_{k}]})}}
d\vec{t}.
$$
Further using the same arguments as for the case $\ve_1=\ve_2$ one can check that there exists the finite limit
$$
E(T^{x}_{\ve_1,k})^lT^{x}_{\ve_2,k})^{2p-l}
$$
as $\ve_1,\ve_2\to0.$
\end{proof}
Notice that one-dimensional Wiener process $w(t),\ t\in[0;1]$ is Gaussian integrator generated by identity operator. Then it follows from Theorem \ref{thm2} and Lemma \ref{lem6} that
$$
E\int_{\Delta_{k}}\prod^{k-1}_{i=1}\delta_0(w(t_{i+1})-w(t_i))d\vec{t}=
$$
$$
=\frac{1}{(2\pi)^{\frac{k-1}{2}}}\int_{\Delta_{k}}
\frac
{1}
{\sqrt{G(\1_{[t_1;t_2]},\ldots,\1_{[t_{k-1};t_k]})}}d\vec{t}=
$$
$$
\frac{1}{(2\pi)^{\frac{k-1}{2}}}\int_{\Delta_{k}}
\frac
{1}
{\sqrt{\prod^{k-1}_{i=1}(t_{i+1}-t_i)}}d\vec{t}=
$$
$$
$$
$$
=\dfrac{ 1 }{ 2^{\frac{k-1}{2}} \Gamma \big(  \frac{k+3}{2} \big)   }\; .
$$
Moreover, it follows from Theorem \ref{thm2} that for any $p\in\mbN$ there exists some positive constant $c(k)$ which depends on $k$ such that
$$
E(T^x_k)^{2p}\leq c(k)E(T^y_k)^{2p},
$$
where the process $y$ is defined in \eqref{eq11}.
\section{ It\^{o}-Wiener expansion for the self-intersection local time of Gaussian integrators }
Let $\xi$ be a white noise in the space $L_2([0;1])$ generated by Wiener process $w(t),\ t\in[0;1].$ Consider a one dimensional Gaussian integrator $x(t)=(A1_{[0;t]},\xi),\ t\in [0;1]$ with an operator $A$ which satisfies conditions 1), 2) of Theorem \ref{thm2}. For the process $x$ there exists the self-intersection local time $T_k^x.$ The aim of this section is to find It\^{o}-Wiener expansion of random variable $T_k^x.$ For this purpose, we will use Fourier-Wiener transform. Let us start from general definitions of It\^{o}-Wiener expansion and Fourier-Wiener transform. Let $\alpha$ be a square integrable random variable which is measurable with respect to a white noise $\xi.$
\begin{thm}(It\^{o}-Wiener expansion)\cite{15}
\label{thm4}
The random variable $\alpha$ can be uniquely represented as the convergent in mean square series of orthogonal summands
$$
\alpha=\sum^{\infty}_{k=0}\int_{\Delta_k}a_k(t_1,\ldots,t_k)dw(t_1)\ldots dw(t_k),
$$
where
$$
a_0=E\alpha,
$$
$$
\int_{\Delta_k}a^2_k(t_1,\ldots,t_k)dt_1\ldots dt_k<+\infty,\ k\geq1.
$$
Moreover,
$$
E\alpha^2=\sum^{\infty}_{k=0}\int_{\Delta_k}a^2_k(t_1,\ldots,t_k)dt_1\ldots dt_k.
$$
\end{thm}
Some times for It\^{o}-Wiener expansion of random variable $\alpha$ is convenient to use general notation
$$
\alpha=\sum^{\infty}_{k=0}A_k(\xi,\ldots,\xi),
$$
where symmetric Hilbert-Schmidt forms $A_k$ on $L_2(0;1)^{\otimes k}$ have representation
$$
A_k(\xi,\ldots,\xi)=\int_{\Delta_k}a_k(t_1,\ldots,t_k)dw(t_1)\ldots dw(t_k).
$$
\begin{defn}\cite{15}
\label{defn3.1} The stochastic derivative of $\alpha$   is a
square integrable random element $D\alpha$  in $L_2(0;1)$ such that for
every $h\in L_2(0;1)$
$$
(D\alpha, h)=\sum^\infty_{k=0}kA_k(h, \xi, \ldots, \xi).
$$
\end{defn}
\begin{lem}\cite{15}
\label{lem1.5.5} The random variable $\alpha$ has a stochastic derivative iff
$$
\sum^\infty_{k=0}k\cdot k!\|A_k\|^2_k<+\infty.
$$
\end{lem}
Skorokhod integral $I$ can be defined as the adjoin operator $I=D^{\ast}$ acting from the space of square-integrable $L_2(0;1)$-valued random elements to the space of square integrable random variables. In terms of It\^{o}-Wiener expansion,  Skorokhod integral can be described as follows. Suppose that the square-integrable random element $\alpha$ can be represented by the series
$$
(\alpha,h)=\sum^\infty_{k=0}A_k(h, \xi,\ldots,\xi),\ h\in L_2(0;1).
$$
It can be checked that $A_k\in L_2(0;1)^{\otimes k+1}$ for every $k\geq 0.$ Denote
by $\Lambda A_k$  its symmetrization. Then
\begin{defn}\cite{15}
\label{defn3.2} The random variable
$$
I(\alpha)=\sum^{\infty}_{k=0}\Lambda A_k(\xi,\ldots,\xi).
$$
is said to be Skorokhod integral of $\alpha.$
\end{defn}
As the tool for investigation of square integrable random variables which are measurable with respect to a white noise $\xi$ we will use Fourier-Wiener transform. Let us recall the definition.
\begin{defn}\cite{14}
\label{defn3}
The functional
$$
{\mathcal T}(\alpha)(h)=E\alpha\exp\{(h,\xi)-\frac{1}{2}\|h\|^2\}
$$
is said to be Fourier-Wiener transform of random variable $\alpha.$
\end{defn}
One can check that It\^{o}-Wiener expansion of $\exp\{(h,\xi)-\frac{1}{2}\|h\|^2\}$ has the following representation
$$
\exp\{(h,\xi)-\frac{1}{2}\|h\|^2\}=\sum^{\infty}_{n=0}\int_{\Delta_n}h(t_1)\ldots h(t_n)dw(t_1)\ldots dw(t_n).
$$
Then for random variable
$$
\alpha=\sum^{\infty}_{n=0}\int_{\Delta_n}a_n(t_1,\ldots,t_n)dw(t_1)\ldots dw(t_n)
$$
Taylor expansion of
Fourier--Wiener transform of random variable $\alpha$ has the following representation
\begin{equation}
\label{eq201}
{\mathcal T}(\alpha)(h)=\sum^{\infty}_{n=0}\int_{\Delta_n}a_n(t_1,\ldots,t_n)h(t_1)\ldots h(t_n)dt_1\ldots dt_n.
\end{equation}
Therefore, Taylor expansion \eqref{eq201} of
Fourier--Wiener transform of random variable $\alpha$ uniquely determines kernels $\{a_n,\ n\geq0\}$ of its It\^{o}-Wiener expansion \cite{6}, \cite{7}.
In this section we will keep previous notations. We suppose that $G(e_1, \ldots, e_n)$ is Gram determinant constructed from elements $e_1,\ldots, e_n$ and $B(e_1, \ldots, e_n)$ is corresponding Gramian matrix. Denote by $ P_{t_1\ldots t_k}$ an orthogonal projection onto a linear subspace generated by elements $A\mathbf{1}_{[ t_{1};t_{2}]},\cdots ,A\mathbf{1}_{[ t_{k-1};t_{k}] }.$ It can be checked \cite{6}, \cite{7}  that for every $h \in L_2([0,1]) $ the following relation holds
$$
\mathcal{T}(T_k^x )(h) =  \int_{\Delta _k } \dfrac{e^{-\frac{1}{2}\|P_{t_1\ldots t_k}h \| ^{2}}}{(2\pi )^{\frac{k-1}{2}}\sqrt{G(A\mathbf{1}_{[ t_{1};t_{2}]},\cdots ,A\mathbf{1}_{[ t_{k-1};t_{k}]} )} } \; d\vec{t}=
$$
\begin{equation}
\label{eq13}
= \sum_{n=0}^{\infty} \dfrac{(-1)^n}{2^n n!(2\pi )^{\frac{k-1}{2}}} \int_{\Delta _k } \dfrac{\|P_{t_1\ldots t_k}h \| ^{2n}}{\sqrt{G(A\mathbf{1}_{[ t_{1};t_{2}]},\cdots ,A\mathbf{1}_{[ t_{k-1};t_{k}]} )} } \; d\vec{t}.
\end{equation}
It was proved in \cite{7} that for $h\in L_2([0,1])$
$$
P_{t_1\ldots t_k}h =\frac{1}{G(A\1_{[t_1;t_2]},\ldots,A\1_{[t_{k-1};t_k]})}\sum^{k-1}_{ij=1}(-1)^{i+j}M_{ij}(A\1_{[t_i;t_{i+1}]},h)\1_{[t_j;t_{j+1}]},
$$
where $M_{ij}$ is the minor of matrix
$$
B(A\1_{[t_1;t_2]},\ldots,A\1_{[t_{k-1};t_k]})
$$
corresponding to $i$-th row and the $j$-th column. Then
$$
\|P_{t_1\ldots t_k}h \|^{2n}=
$$
$$
\frac{1}{G(A\1_{[t_1;t_2]},\ldots,A\1_{[t_{k-1};t_k]})^{2n}}\cdot
$$
$$
\cdot \sum^{k-1}_{i^1_1,j^1_1,i^1_2,j^1_2,\ldots,i^n_1,j^n_1,i^n_2,j^n_2=1}\prod^{n}_{l=1}(-1)^{i^l_1+j^l_1+i^l_2+j^l_2}
M_{i^l_1j^l_1}M_{i^l_2j^l_2}(\1_{[t_{j^l_1};t_{j^l_1+1}]},\1_{[t_{j^l_2};t_{j^l_2+1}]})\cdot
$$
$$
\cdot(A\1_{[t_{i^l_1};t_{i^l_1+1}]},h)(A\1_{[t_{i^l_2};t_{i^l_2+1}]},h)=
$$
$$
\frac{1}{G(A\1_{[t_1;t_2]},\ldots,A\1_{[t_{k-1};t_k]})^{2n}}\cdot
$$
$$
\cdot \sum^{k-1}_{i^1_1,j^1_1,i^1_2,j^1_2,\ldots,i^n_1,j^n_1,i^n_2,j^n_2=1}\prod^{n}_{l=1}(-1)^{i^l_1+j^l_1+i^l_2+j^l_2}
M_{i^l_1j^l_1}M_{i^l_2j^l_2}(\1_{[t_{j^l_1};t_{j^l_1+1}]},\1_{[t_{j^l_2};t_{j^l_2+1}]})\cdot
$$
\begin{equation}
\label{eq403}
\int^{1}_{0}\int^{1}_{0}A\1_{[t_{i^l_1};t_{i^l_1+1}]}(s^l_1)A\1_{[t_{i^l_2};t_{i^l_2+1}]}(s^l_2)h^{\otimes 2}(\vec{s^l})d\vec{s^l}.
\end{equation}
Here we use notation $h^{\otimes n}(\vec{s})$ for $h(s_1)\ldots h(s_n).$
It follows from \eqref{eq403} that \eqref{eq13} equals
$$
\sum_{n=0}^{\infty}\int_{\Delta_{2n}} \dfrac{(-1)^n}{(2\pi )^{\frac{k-1}{2}}}\int_{\Delta _k }\frac{1}{G(A\1_{[t_1;t_2]},\ldots,A\1_{[t_{k-1};t_k]})^{2n+\frac{1}{2}}}\cdot
$$
$$
\cdot \sum^{k-1}_{i^1_1,j^1_1,i^1_2,j^1_2,\ldots,i^n_1,j^n_1,i^n_2,j^n_2=1}\prod^{n}_{l=1}(-1)^{i^l_1+j^l_1+i^l_2+j^l_2}
M_{i^l_1j^l_1}M_{i^l_2j^l_2}(\1_{[t_{j^l_1};t_{j^l_1+1}]},\1_{[t_{j^l_2};t_{j^l_2+1}]})\cdot
$$
\begin{equation}
\label{eq404}
\cdot A\1_{[t_{i^l_1};t_{i^l_1+1}]}(s^l_1)A\1_{[t_{i^l_2};t_{i^l_2+1}]}(s^l_2)d\vec{t}\ h^{\otimes 2n}(\vec{s^l})d\vec{s^l}.
\end{equation}
It follows from \eqref{eq404} that kernels of It\^{o}-Wiener expansion for $\mathcal{T}(T_k^x )(h)$ have the representation
$$
a_{2n}=\dfrac{(-1)^n}{(2\pi )^{\frac{k-1}{2}}}\int_{\Delta _k }\frac{1}{G(A\1_{[t_1;t_2]},\ldots,A\1_{[t_{k-1};t_k]})^{2n+\frac{1}{2}}}\cdot
$$
$$
\cdot \sum^{k-1}_{i^1_1,j^1_1,i^1_2,j^1_2,\ldots,i^n_1,j^n_1,i^n_2,j^n_2=1}\prod^{n}_{l=1}(-1)^{i^l_1+j^l_1+i^l_2+j^l_2}
M_{i^l_1j^l_1}M_{i^l_2j^l_2}(\1_{[t_{j^l_1};t_{j^l_1+1}]},\1_{[t_{j^l_2};t_{j^l_2+1}]})\cdot
$$
$$
\cdot A\1_{[t_{i^l_1};t_{i^l_1+1}]}(s^l_1)A\1_{[t_{i^l_2};t_{i^l_2+1}]}(s^l_2)d\vec{t}. $$
Consider some examples. Let $A$ be an operator of multiplication by some measurable function $\phi.$ Suppose that there exist constants $m,\  M >0 $ such that for every $ r \in [0;1] $
$$
m \leqslant \; |\phi (r)| \; \leqslant \; M.
$$
This condition guarantees a continuous invertibility of operator $A.$ It follows from Theorem \ref{thm2} that for a Gaussian integrator generated by the given operator $A$ there exists the $k$-multiple self-intersection local time $T_k^x.$ Let us find It\^{o}-Wiener expansion of random variable $T_k^x.$
\begin{lem}
\label{lem7}
It\^{o}-Wiener expansion of $k$-multiple self-intersection local time of  Gaussian integrator $x$ associated with an operator of multiplication by the function $ \phi $ has the representation
\begin{equation}
\label{eq14}
T_k^x = \E T_k^x + \sum_{n=1}^{\infty } \dfrac{(-1)^n}{(2\pi )^{\frac{k-1}{2}}} \int _{\Delta _{2n}} \phi ^{\bigotimes 2n}(\vec{s})\; b_{2n}(\vec{s}) dw(s_1)\ldots dw(s_{2n})
\end{equation}

where $ \E T_k^x $ and $ b_{2n} $ are defined as follows
\begin{equation}
\label{eq24}
\E T_k^x = \dfrac{1}{(2\pi )^{\frac{k-1}{2}}} \int_{\Delta _k } \dfrac{d\vec{t}}{\sqrt{\prod _{i=1} ^{k-1} \int^{t_{i+1}}_{t_i} \phi ^2(r)dr}  }
\end{equation}
and
\begin{equation}
\label{eq15}
b_{2n}(\vec{s} )= \int_{\Delta _k } \dfrac{d\vec{t}}{\sqrt{\prod _{i=1} ^{k-1} \int^{t_{i+1}}_{t_i} \phi ^2(r)dr} } \sum_{1\leqslant i_1,\ldots ,i_n \leqslant k-1}  \dfrac{1}{ \prod_{j=1}^{n} \int^{t_{i_{j}+1}}_{t_{i_j}} \phi ^2(r)dr}  \mathbf{1}_{[t_{i_1}; t_{i_{1}+1} ]^2\times \cdots \times [t_{i_n}; t_{i_{n}+1} ]^2} (\vec{s})  \; .
\end{equation}
\end{lem}

\begin{proof} To prove the lemma we will use the connection between Fourier-Wiener transform and Ito-Wiener expansion. As it was mentioned above, it suffices to find Taylor's expansion of Fourier-Wiener transform of self-intersection local time.
One can check that
$$ \|P_{t_1\ldots t_k}h \| ^{2}= \sum_{i=1}^{k-1}\dfrac{\Big( \int^{t_{i+1}}_{t_i}\phi (r) h(r)dr \Big) ^2}{\int^{t_{i+1}}_{t_i} \phi ^2(r)dr}  $$
and
$$  G\big( A\mathbf{1}_{[ t_{1};t_{2}]},\cdots ,A\mathbf{1}_{[ t_{k-1};t_{k}]} \big) = \prod _{i=1} ^{k-1} \int^{t_{i+1}}_{t_i} \phi ^2 (r)dr . $$
Then \eqref{eq13} equals
$$
\dfrac{(-1)^n}{2^n n!(2\pi )^{\frac{k-1}{2}}} \int_{\Delta _k } \dfrac{d\vec{t}}{\sqrt{\prod _{i=1} ^{k-1} \int^{t_{i+1}}_{t_i} \phi ^2(r)dr} } \sum_{1\leqslant i_1,\ldots ,i_n \leqslant k-1} \prod_{j=1}^{n} \dfrac{\big( \int^{t_{i_{j}+1}}_{t_{i_j}}\phi (r) h(r)dr \big) ^2}{\int^{t_{i_{j}+1}}_{t_{i_j}} \phi ^2(r)dr}  =   $$

$$ =   \dfrac{(-1)^n}{2^n n!(2\pi )^{\frac{k-1}{2}}}\int_{\Delta _k } \dfrac{1}{\sqrt{\prod _{i=1} ^{k-1} \int^{t_{i+1}}_{t_i} \phi ^2(r)dr} } \times $$
$$ \times \left( \sum_{1\leqslant i_1,\ldots ,i_n \leqslant k-1} \int _{[0,1]^{2n}} h^{\bigotimes 2n}(\vec{s})\phi ^{\bigotimes 2n}(\vec{s}) \dfrac{1}{ \prod_{j=1}^{n} \int^{t_{i_{j}+1}}_{t_{i_j}} \phi ^2(r)dr}  \mathbf{1}_{[t_{i_1}; t_{i_{1}+1} ]^2\times \cdots \times [t_{i_n}; t_{i_{n}+1} ]^2} (\vec{s}) d\vec{s} \right) \; d\vec{t} \; .   $$
\end{proof}
Let us focus on Brownian motion case. This example can be derived from the previous example by taking $ \phi $ identically equals one. Then it follows from \eqref{eq14} -- \eqref{eq15} that
$$
T_k^w = \E T_k^w + \sum_{n=1}^{\infty } \dfrac{(-1)^n}{(2\pi )^{\frac{k-1}{2}}} \int _{\Delta _{2n}} \; b_{2n}(\vec{s}) dw(s_1)\ldots dw(s_{2n}),
$$
where kernels $b_{2n}$ are defined by
$$
 b_{2n}(\vec{s} )= \int_{\Delta _k } \dfrac{d\vec{t}}{\sqrt{\prod _{i=1} ^{k-1} (t_{i+1}-{t_i})}} \sum_{1\leqslant i_1,\ldots ,i_n \leqslant k-1}  \dfrac{1}{ \prod_{j=1}^{n} (t_{i_{j}+1}-t_{i_j})}  \mathbf{1}_{[t_{i_1}; t_{i_{1}+1} ]^2\times \cdots \times [t_{i_n}; t_{i_{n}+1} ]^2} (\vec{s}) \; ,
$$
\begin{equation}
\label{eq25}
\E T_k^w = \dfrac{1}{(2\pi )^{\frac{k-1}{2}}} \int_{\Delta _k } \dfrac{d\vec{t}}{\sqrt{\prod _{i=1} ^{k-1} (t_{i+1}-t_i)}  }.
\end{equation}
The square norm of self intersection local time of Wiener process is given by
\begin{equation}
\label{eq16}
\E (T_k^w )^2 =  \Big( \E T_k^w \Big )^2    +  \sum_{n=1}^{\infty } \dfrac{(2n!) ^2}{(n!) ^2 2^{2n}(2\pi )^{k-1}} \int _{\Delta _{2n}} b_{2n} ^2(\vec{s}) d\vec{s}.
\end{equation}
Applying Lemma \ref{lem6} one can conclude
\begin{corollary}
\label{corollary1}
A square norm of self intersection local time of Wiener process has the form
$$
\E (T_k^w )^2 = \dfrac{1}{ 2^{k-1} \Gamma \big(  \frac{k+3}{2} \big)^2  }\; + $$

\begin{equation}
\label{eq19}
  +\sum_{n=1}^{\infty } \dfrac{(2n)!}{n! ^2 2^{2n}(2\pi )^{k-1}} \int_{\Delta _k ^{2} } \; \dfrac{ d\vec{t} \;  d\vec{t'}}{\sqrt{\prod _{i=1} ^{k-1}( t_{i+1}-t_i)( t'_{i+1}-t'_i)} }\left[ \; \sum_{ 1 \leqslant i,j \leqslant k-1 } \;\dfrac{ \lambda \; \big( [t_{i};t_{i+1}]\cap [t'_{j};t'_{j+1}] \big)^2 }{( t_{i+1} - t_{i} )\; ( t'_{j+1}- t'_{j})} \right]  ^n \;.
\end{equation}
Here $\lambda $ is Lebesgue measure.
\end{corollary}

\begin{proof}
Note that
$$  \int _{\Delta _{2n}} b_{2n} ^2 (\vec{s})d\vec{s} =  \dfrac{1}{2n \; !}\int_{\Delta^2 _k} \dfrac{ d\vec{t} \;  d\vec{t'}}{\sqrt{\prod _{i=1} ^{k-1}( t_{i+1}-t_i)( t'_{i+1}-t'_i)} } \sum_{\substack{ 1 \leqslant i_1,\ldots ,i_n \leqslant k-1 \\ 1 \leqslant j_1,\ldots ,j_n \leqslant k-1}} $$
$$ \int _{[0,1]^{2n}} \dfrac{\mathbf{1}_{[t_{i_1}; t_{i_{1}+1} ]^2}(s_1,s_2)\mathbf{1}_{[t'_{j_1}; t'_{j_{1}+1} ]^2}(s_1,s_2) \cdots \mathbf{1}_{[t_{i_n};t_{i_{n}+1} ]^2}(s_{2n-1},s_{2n})\mathbf{1}_{[t'_{j_n}; t'_{j_{n}+1} ]^2}(s_{2n-1},s_{2n})}{( t_{i_{1}+1} - t_{i_1} )\cdots ( t'_{j_{n}+1}- t'_{j_n})}  d\vec{s} \; = $$
$$   =  \dfrac{1}{2n \; !} \int_{\Delta _k \times \Delta _k } \; \dfrac{ d\vec{t} \;  d\vec{t'}}{\sqrt{\prod _{i=1} ^{k-1}( t_{i+1}-t_i)( t'_{i+1}-t'_i)} } \; \sum_{\substack{ 1 \leqslant i_1,\ldots ,i_n \leqslant k-1 \\ 1 \leqslant j_1,\ldots ,j_n \leqslant k-1}} \;  $$
$$\; \prod_{r=1}^n \; \dfrac{  \; \lambda \; \big( [t_{i_r};t_{i_{r}+1}]\cap [t'_{i_r};t'_{i_{r}+1}] \big) ^2 }{( t_{i_{r}+1} - t_{i_r} )\; ( t'_{j_{r}+1}- t'_{j_r})} \; = $$
$$ = \;\dfrac{1}{2n \; !} \int_{\Delta _k \times \Delta _k } \; \dfrac{ d\vec{t} \;  d\vec{t'}}{\sqrt{\prod _{i=1} ^{k-1}( t_{i+1}-t_i)( t'_{i+1}-t'_i)} } \;\left[ \; \sum_{ 1 \leqslant i,j \leqslant k-1 } \;\dfrac{ \lambda \; \big( [t_{i};t_{i+1}]\cap [t'_{j};t'_{j+1}] \big)^2 }{( t_{i+1} - t_{i} )\; ( t'_{j+1}- t'_{j})} \right]  ^n \;  $$
which ends the proof.
\end{proof}
Let us estimate a rate of convergence of series \eqref{eq16} in the case of double self intersection local time. To do this we need the following statement which was proved in \cite{11}.

\begin{lem}\cite{11}
\label{lem7}
There exist a constant $c > 0 $ such that for all integers $ n\geqslant 2 $
\begin{equation}
\label{20}
 \int_{\Delta _2 ^{2} } \; \dfrac{1}{\sqrt{( t_{2}-t_1)( t'_{2}-t'_1)} }\left[  \;\dfrac{ \lambda \; \big( [t_{1};t_{2}]\cap [t'_{1};t'_{2}] \big)^2 }{( t_{2} - t_{1} )\; ( t'_{2}- t'_{1})} \right]  ^n d\vec{t} \;  d\vec{t'} \; \leqslant \; \dfrac{c}{n^2}.
 \end{equation}
\end{lem}
Using Stirling's formula one can check that
\begin{equation}
\label{eq26}
\dfrac{(2n)! }{(n! 2^{n})^2 }\leq\dfrac{1}{\sqrt{n}}. \;
\end{equation}
Lemma \ref{lem7} implies
\begin{equation}
\label{eq27}
 \dfrac{(2n)! }{n! ^2 2^{2n}2\pi } \int_{\Delta _2 ^{2} } \; \dfrac{1}{\sqrt{( t_{2}-t_1)( t'_{2}-t'_1)} }\left[  \;\dfrac{ \lambda \; \big( [t_{1};t_{2}]\cap [t'_{1};t'_{2}] \big)^2 }{( t_{2} - t_{1} )\; ( t'_{2}- t'_{1})} \right]  ^n \;d\vec{t} \;  d\vec{t'}\leq\ \frac{c}{n^{\frac{5}{2}}},\ c>0.
\end{equation}
It follows from \eqref{eq27} that the series
$$ \sum _{n\geqslant 1} n. \dfrac{(2n)! }{n! ^2 2^{2n}2\pi } \int_{\Delta _2 ^{2} } \; \dfrac{1}{\sqrt{( t_{2}-t_1)( t'_{2}-t'_1)} }\left[  \;\dfrac{ \lambda \; \big( [t_{1};t_{2}]\cap [t'_{1};t'_{2}] \big)^2 }{( t_{2} - t_{1} )\; ( t'_{2}- t'_{1})} \right]  ^n d\vec{t} \;  d\vec{t'}$$
converges, which means that the self-intersection local time $T^w_2$ is stochastically differentiable \cite{18}.
\section{ Clark representation for the self-intersection local time of Gaussian integrators }
Consider a one dimensional Gaussian integrator $x(t)=(A1_{[0;t]},\xi),\ t\in [0;1].$  Suppose that an operator $A$ satisfies assumptions of Theorem \ref{thm2}. The aim of this section is to establish Clark formula for
$$
T^x_k=\int_{\Delta_k}\prod^{k-1}_{i=1}\delta_0(x(t_{i+1})-x(t_i))d\vec{t}
$$
To do this we will consider
$$
\prod^{k-1}_{i=1}\delta_0(x(t_{i+1})-x(t_i))
$$
as generalized functional of white noise $\xi.$ Let us start from the notion of generalized Gaussian functionals. Here we follow articles \cite{7}, \cite{214}, \cite{215}.
As before we suppose that $\xi$ is a white noise in $L_2([0; 1]).$ Assume that $\sigma$-field of random events $\cF$ is generated by $\xi.$ For $F\in L_2(\Omega, \cF, P)$ denote by $F=\sum^{\infty}_{n=0}I_n(f_n)$ its It\^{o}-Wiener expansion. For $\alpha\in\mbR$ let us define Sobolev space of Gaussian functionals of order $\alpha$  by introducing the norm
\begin{equation}
\label{eq2.3.1'}
\|F\|^2_{2, \alpha}=\sum^N_{n=0}(1+n)^\alpha\|I_n(f_n)\|_2
\end{equation}
on the space
$$
\cP=\left\{F\in L_2(\Omega, \cF, P): \ F=\sum^N_{n=0}I_n(f_n), \ \ N\geq1\right\}
$$
of random variables with finite It\^o--Wiener expansion.
\begin{defn}
\label{defn4.1}\cite{213}

Completion of $\cP$ with respect to $\|\cdot\|_{2, \alpha}$ is said to be  Sobolev space $D_{2, \alpha}$ of order $\alpha$.
\end{defn}

It follows from the  definition that for $\alpha_1<\alpha_2$ \ $D_{2, \alpha_1}\supset D_{2, \alpha_2}.$
Put
$$
D^\infty:=\mathop{\cap}\limits_{\alpha>0}D_{2, \alpha},
$$
$$
D^{-\infty}:=\mathop{\cup}\limits_{\alpha>0}D_{2, -\alpha}.
$$
Denote by $D^\infty(\mbR^d)$ the space of random vectors with coordinates from $D^\infty.$
Since $D_{2,0}=L_2(\Omega, \cF, P),$ then for $\alpha\geq0$ elements of spaces $D_{2, \alpha}$ are ``classical'' Gaussian functionals. In the case $\alpha<0$ the elements of $D_{2,\alpha}$ in general can not be considered as  random variables.
\begin{defn}
\label{defn4.2}\cite{213}
 Elements of spaces $D_{2, -\alpha},\ \alpha>0$ are said to be generalized functionals of white noise (generalized Gaussian functionals) \cite{213}.
\end{defn}
Examples of  generalized Gaussian functionals can be obtained as a result of action of Schwartz distributions on the elements of $D_{2,\alpha},\  \alpha>0.$
Let $\cS(\mbR^d)$ be Schwartz space of rapidly decreasing $C^\infty$-functions on $\mbR^d.$ For $k\in\mbZ $ set
\begin{equation}
\label{eq2.3.3}
\|\vf\|_{2k}=\|(1+|x|^2-\Delta)^k\vf\|_{\infty}\,,
\end{equation}
$\vf\in\cS(\mbR^d),  \ \|f\|_\infty=\sup_{x\in\mbR^d}|f(x)|,$ where $\Delta=\sum^d_{i=1}\left(\dfrac{\pt}{\pt x_i}\right)^2.$ Suppose that $\cG_{2k}$ is a completion of $\cS(\mbR^d)$ with respect to \eqref{eq2.3.3}. It is known, that
$$
\cS(\mbR^d)\subset\ldots\subset \cG_2\subset \cG^0=\wh{C}(\mbR^d)\subset \cG_{-2}\subset\ldots\subset\cS'(\mbR^d),
$$
where $\wh{C}(\mbR^d)$ Banach space of all
continuous functions on $\mbR^d$ tending to 0 at infinity, endowed with the supremum norm and $\cS'(\mbR^d)$ is Schwartz space of tempered distributions on $\mbR^d$ \cite{213}. For $F\in D^\infty(\mbR^d)$ denote by $\sigma:=((DF_i, DF_j))^d_{ij=i}$ Gramian matrix for random elements $DF_1, \ldots, DF_d.$ Here $D$ is a stochastic derivative \cite{18}. Assume that\vskip10pt

1) $\det\sigma>0$\vskip10pt

2) $[\det\sigma]^{-1}\in\cap_{1<p<\infty}L_p(\Omega,\cF,P).$\vskip10pt

\noindent
Then the following statement holds.
\begin{thm}\cite{213}
\label{thm4.1}
For every $p\in(1, +\infty)$ and $k=0, 1, 2, \ldots$ there exists a positive constant $c=c_{p, k}$ such that for all $\vf\in\cS(\mbR^d)$ the following relation holds
$$
\|\vf(F)\|_{p, -2k}\leq c\|\vf\|_{-2k}.
$$
\end{thm}
Notice that for $k_0\geq1$ and $\vf\in \cG_{-2k_0}$ there exist $\{\vf_\ve\}_{\ve>0}\in\cS(\mbR^d)$ such that $\vf_\ve\overset{\cG_{-2k_0}}{\longrightarrow\vf},$ when $\ve\to0.$
Theorem \ref{thm4.1} implies, that
$$
\|\vf_{\ve_1}(F)-\vf_{\ve_2}(F)\|_{p, -2k_0}\leq c_{p, k}\|\vf_{\ve_1}-\vf_{\ve_2}\|_{-2k_0}\to0
$$
 as $\ve_1, \ \ve_2\to0.$ Consequently, $\{\vf_\ve(F)\}_{\ve>0}$ is fundamental in $D_{2, -2k_0},$ i.e. there exists the limit of $\vf_\ve(F),$ when $\ve\to0$ in $D_{2, -2k_0}.$

\begin{defn}
\label{defn4.2}
The value of a generalized function $\vf$ on $F$ is defined as
$$
\vf(F): =\lim_{\ve\to0}\vf_{\ve}(F).
$$
\end{defn}

Under the condition of Theorem \ref{thm4.1} the  generalized functional $\vf(F)$ has a formal It\^o--Wiener expansion. The elements of this expansion are the limits of corresponding terms in expansion of $\vf_\ve(F).$ In particular, it is natural to define expectation $E\vf(F)$ as the limit of $E\vf_\ve(F).$ Since that
$$
E\delta_0(w(t)-w(s))=\lim_{\ve\to0}Ef_{\ve}(w(t)-w(s))=
$$
$$
\lim_{\ve\to0}\int_{\mbR}f_{\ve}(y)\frac{1}{\sqrt{2\pi}}\frac{1}{\sqrt{t-s}}e^{-\frac{y^2}{2(t-s)}}dy=
$$
$$
=\frac{1}{\sqrt{2\pi}}\frac{1}{\sqrt{t-s}},\ \ve\to0.
$$
Let us find the Fourier--Wiener transform of
$$
\prod^{k-1}_{i=1}\delta_0(x(t_{i+1})-x(t_i)).
$$
Suppose that $G(A\1_{[t_1;t_2]},\ldots,A\1_{[t_{k-1};t_k]})\neq 0,$ then it follows from Theorem \ref{thm4.1} that
$$
\prod^{k-1}_{i=1}\delta_0(x(t_{i+1})-x(t_i))
$$
is the generalized functional from $\xi.$ It's Fourier--Wiener transform may be obtained as follows
$$
\cT\left(\prod^{k-1}_{i=1}\delta_0(x(t_{i+1})-x(t_i))\right)(h)=
\lim_{\ve\to0}\cT\left(\prod^{k-1}_{i=1}f_\ve(x(t_{i+1})-x(t_i))\right)(h)=
$$
$$
=\lim_{\ve\to0}E\int_{\Delta_k}\prod^{k-1}_{i=1}f_\ve(x(t_{i+1})-x(t_i))d\vec{t}\ \cE(h)=
$$
\begin{equation}
\label{eq2.3.4}
=\lim_{\ve\to0}E\int_{\Delta_k}\prod^{k-1}_{i=1}f_\ve((A\1_{[t_i;t_{i+1}]},\xi)+(A\1_{[t_i;t_{i+1}]},h))d\vec{t}.
\end{equation}
Here
$$
\cE(h)=e^{(h,\xi)-\frac{1}{2}\|h\|^2}.
$$
It can be checked \cite{7} that \eqref{eq2.3.4} equals

$$
\dfrac{e^{-\frac{1}{2}\|P_{t_1\ldots t_k}h \| ^{2}}}{(2\pi )^{\frac{k-1}{2}}\sqrt{G(A\mathbf{1}_{[ t_{1};t_{2}]},\cdots ,A\mathbf{1}_{[ t_{k-1};t_{k}]} )} },
$$
where $P_{t_1\ldots t_k}$ is a projection on the linear subspace generated by $A\mathbf{1}_{[ t_{1};t_{2}]},\cdots ,A\mathbf{1}_{[ t_{k-1};t_{k}]}.$
For Wiener process $A=I.$ It implies that
$$
\cT\left(\prod^{k-1}_{i=1}\delta_0(w(t_{i+1})-w(t_i))\right)(h)=
\lim_{\ve\to0}\cT\left(\prod^{k-1}_{i=1}f_\ve(w(t_{i+1})-w(t_i))\right)(h)=
$$
$$
=\dfrac{e^{-\frac{1}{2}\|P_{t_1\ldots t_k}h \| ^{2}}}{(2\pi )^{\frac{k-1}{2}}\sqrt{G(\mathbf{1}_{[ t_{1};t_{2}]},\cdots ,\mathbf{1}_{[ t_{k-1};t_{k}]} )} },
$$
where $P_{t_1\ldots t_k}$ is a projection on the linear subspace generated by $\mathbf{1}_{[ t_{1};t_{2}]},\cdots ,\mathbf{1}_{[ t_{k-1};t_{k}]}.$
Since indicators of disjoint sets are orthogonal, then
$$
\dfrac{e^{-\frac{1}{2}\|P_{t_1\ldots t_k}h \| ^{2}}}{(2\pi )^{\frac{k-1}{2}}\sqrt{G(\mathbf{1}_{[ t_{1};t_{2}]},\cdots ,\mathbf{1}_{[ t_{k-1};t_{k}]} )} }=
$$
$$
=\frac{e^{-\frac{1}{2}\sum^{k-1}_{i=1}\|P_{t_it_{i+1}}h \| ^{2}}}{(2\pi )^{\frac{k-1}{2}}\prod^{k-1}_{i=1}\sqrt{t_{i+1}-t_i}}=
$$
$$
=\prod^{k-1}_{i=1}\frac{e^{-\frac{1}{2}\|P_{t_it_{i+1}}h \| ^{2}}}{\sqrt{(2\pi )}\sqrt{t_{i+1}-t_i}}=
$$
$$
=\prod^{k-1}_{i=1}\lim_{\ve\to0}\cT\Big(f_\ve(w(t_{i+1})-w(t_i))\Big)(h)=
$$
$$
=\prod^{k-1}_{i=1}\cT\Big(\delta_0(w(t_{i+1})-w(t_i))\Big)(h).
$$
Notice that in general case
$$
\cT\Big(\int_{\Delta_k}\prod^{k-1}_{i=1}\delta_0(x(t_{i+1})-x(t_i))d\vec{t}\Big)(h)=
$$
$$
=\lim_{\ve\to0}\cT\Big(\int_{\Delta_k}\prod^{k-1}_{i=1}f_\ve(x(t_{i+1})-x(t_i))d\vec{t}\Big)(h)=
$$
$$
=\lim_{\ve\to0}\int_{\Delta_k}\cT\Big(\prod^{k-1}_{i=1}f_\ve(x(t_{i+1})-x(t_i))\Big)(h)d\vec{t}.
$$
Since
$$
\cT\Big(\prod^{k-1}_{i=1}f_\ve(x(t_{i+1})-x(t_i))\Big)(h)=
$$
$$
=E\prod^{k-1}_{i=1}f_\ve(x(t_{i+1})-x(t_i)+(A^\ast h,\1_{[t_i;t_{i+1}]}))\leq
$$
$$
\leq E\prod^{k-1}_{i=1}f_\ve(x(t_{i+1})-x(t_i))\leq
$$
$$
\leq E\prod^{k-1}_{i=1}\delta_0(x(t_{i+1})-x(t_i))=
$$
$$
=\frac{1}{G(A\1_{[t_1;t_2]},\ldots,A\1_{[t_{k-1};t_k]})}
$$
and we already proved that
$$
\int_{\Delta_k}\frac{1}{G(A\1_{[t_1;t_2]},\ldots,A\1_{[t_{k-1};t_k]})}d\vec{t}<+\infty,
$$
then Lebesgue's dominated convergence theorem implies that
$$
\cT\Big(\int_{\Delta_k}\prod^{k-1}_{i=1}\delta_0(x(t_{i+1})-x(t_i))d\vec{t}\Big)(h)=
$$
$$
=\lim_{\ve\to0}\cT\Big(\int_{\Delta_k}\prod^{k-1}_{i=1}f_\ve(x(t_{i+1})-x(t_i))d\vec{t}\Big)(h)=
$$
$$
=\int_{\Delta_k}\lim_{\ve\to0}\cT\Big(\prod^{k-1}_{i=1}f_\ve(x(t_{i+1})-x(t_i))\Big)(h)d\vec{t}=
$$
$$
=\int_{\Delta_k}\cT\Big(\prod^{k-1}_{i=1}\delta_0(x(t_{i+1})-x(t_i))\Big)(h)d\vec{t}.
$$
Next statement deals with Clark formula for generalized functionals and the relation is treated in a sense of Fourier-Wiener transform.
Suppose that $0 \;< \; s\; <\; t \; <\; 1 $.
\begin{lem}
\label{lem8}
Clark formula for the following generalized functional of Wiener process has the representation
\begin{equation}
\label{11}
 \delta_0 \big( w(t)-w(s) \big) = \dfrac{1}{\sqrt{2\pi (t-s)}} \; + \; \int_s ^t  p'_{t-u} \big( w(u)-w(s) \big) dw(u) \; .
\end{equation}
\end{lem}

\begin{proof}
Suppose that
$$  \delta_0 \big( w(t)-w(s) \big) = \dfrac{1}{\sqrt{2\pi (t-s)}} \; + \; \int_s ^t \eta (u)dw(u) \; .$$
Fourier-Wiener transform of $ \delta_0 \big( w(t)-w(s) \big) $ has the following form \cite{7}
$$ \mathcal{T} \Big( \delta_0 \big( w(t)-w(s) \big)  \Big ) (h) = \dfrac{1}{\sqrt{2\pi (t-s)}} \exp \left\lbrace  -\dfrac{\Big ( \int_s ^t h(r)dr\Big )^2} {2(t-s)}  \right\rbrace  \; = \; p_{t-s} \Big( \int_s ^t h(r)dr \Big )\;,
$$
$h\; \in \; L_2([0;1]) . $
Newton-Leibniz formula
$$
p_t(u)-p_t(0)=\int^{u}_{0}p^{\prime}_{t}(v)dv
$$
implies the relation
\begin{equation}
\label{eq30}
\mathcal{T} \Big( \delta_0 \big( w(t)-w(s) \big)  \Big ) (h) = \; \dfrac{1}{\sqrt{2\pi (t-s)}} \; + \; \int_s ^t p'_{t-s} \Big( \int_s ^\tau h(r)dr \Big )h(\tau) d\tau \; .
\end{equation}
It can be checked that
$$
\mathcal{T} \Big( \delta_0 \big( w(t)-w(s) \big)  \Big ) (h) = \; \dfrac{1}{\sqrt{2\pi (t-s)}} \; + \; \mathcal{T} \left\lbrace  \int_s ^t \eta (u)dw(u) \right\rbrace  (h) \; = \;
$$
\begin{equation}
\label{eq31}
 = \; \dfrac{1}{\sqrt{2\pi (t-s)}} \; + \; \int_s ^t  \mathcal{T} \left\lbrace  \eta (u) \right\rbrace (h)  h(u) du \; .
\end{equation}
It follows from \eqref{eq30} and \eqref{eq31} that
 $$   \mathcal{T} \left\lbrace  \eta (u) \right\rbrace (h) \; = \; p'_{t-s} \Big( \int_s ^u h(r)dr \Big )\; .   $$
Let us check that
$$ \displaystyle \mathcal{T} \left\lbrace  p'_{t-u} \big( w(u)-w(s) \big) \right\rbrace (h)  \; = \; p'_{t-s} \Big( \int_s ^u h(r)dr \Big )\; . $$ Really,
$$ \mathcal{T} \left(  p'_{t-u} \big( w(u)-w(s) \big) \right) (h)  \; = \; \E  p'_{t-u} \left( w(u)-w(s)+\int_s ^u h(r)dr  \right)   =  p'_{t-u} \ast p_{u-s} \left( \int_s ^u h(r)dr  \right)   $$
$$ = \; p'_{t-s} \Big( \int_s ^u h (r)dr\Big )\;  $$
which finishes the proof.
\end{proof}
Lemma \ref{lem8} implies the statement.

\begin{thm}
\label{thm5}
A self-intersection local time of Wiener process has the representation
$$
 T_k^w = \;  \dfrac{1}{2^{\frac{k-1}{2}}\Gamma \Big( \frac{k+3}{2} \Big) }\; + \;   \int_0^1 \beta (\tau) dw(\tau )\; ,
$$
where
$$ \beta(\tau)\; = \;  \sum _{r=1}^{k-1} \sum_{1\leqslant i_1 <\ldots <i_r \leqslant k-1  } \int_{\Delta _k}  \dfrac{1}{\prod _{j\neq i_1, \ldots ,i_r} \sqrt{ 2\pi (t_{j+1}-t_j) }}\int_0^1 \ldots \int_0^1 p'_{t_{i_{r}+1}-\tau} \big( w(\tau)-w(t_{i_r}) \big)\; \cdot$$
$$ \cdot\mathbf{1}_{[ t_{i_r};t_{i_r+1} ]}(\tau) \prod _{l=1}^{r-1}  \left(p'_{t_{i_{l}+1}-\tau_l} \big( w(\tau_l)-w(t_{i_l}) \big)\; \mathbf{1}_{[ t_{i_l};t_{i_l+1} ]}(\tau_l)\right)  dw(\tau_1)\cdots dw(\tau_{r-1})  d\vec{t}  .$$
Here
$$\prod _{l=1}^{r-1}  \left(p'_{t_{i_{l}+1}-\tau_l} \big( w(\tau_l)-w(t_{i_l}) \big)\; \mathbf{1}_{[ t_{i_l};t_{i_l+1} ]}(\tau_l)\right)=1
$$
for $r=1.$
\end{thm}

\begin{proof}
Using Lemma \ref{lem8} one can write

$$  T_k^w =  \int_{\Delta _k} \prod _{i=1}^{k-1} \delta_0 (w(t_{i+1})-w(t_{i}))d\vec{t} = $$
$$ = \int_{\Delta _k} \prod _{i=1}^{k-1} \left(   \dfrac{1}{\sqrt{2\pi (t_{i+1}-t_i)}} + \int_{t_i} ^{t_{i+1}}p'_{t_{i+1}-\tau} \big( w(\tau)-w(t_i) \big) dw(\tau)   \right) d\vec{t} =  $$
$$  = \dfrac{1}{(2\pi )^{\frac{k-1}{2}}} \int_{\Delta _k } \dfrac{d\vec{t}}{\sqrt{\prod _{i=1} ^{k-1}(t_{i+1}-t_i )}  } \; +   $$
$$   + \sum _{r=1}^{k-1} \sum_{1\leqslant i_1 <\ldots <i_r \leqslant k-1  }\int_{\Delta _k}  \dfrac{1}{\prod _{j\neq i_1, \ldots ,i_r} \sqrt{ 2\pi (t_{j+1}-t_j) }} \; \prod _{l=1}^r  \int_{t_{i_l}} ^{t_{i_l +1}} p'_{t_{i+1}-\tau} \big( w(\tau)-w(t_i) \big) dw(\tau)  d\vec{t} $$

$$  = \dfrac{1}{2^{\frac{k-1}{2}}\Gamma \Big( \frac{k+3}{2} \Big) }\; +   $$
$$   + \sum _{r=1}^{k-1} \sum_{1\leqslant i_1 <\ldots <i_r \leqslant k-1  }\int_{\Delta _k}  \dfrac{1}{\prod _{j\neq i_1, \ldots ,i_r} \sqrt{ 2\pi (t_{j+1}-t_j) }} \;  \prod _{l=1}^r  \int_{t_{i_l}} ^{t_{i_l +1}} p'_{t_{i+1}-\tau} \big( w(\tau)-w(t_i) \big)dw(\tau) d\vec{t} \; .$$
Since increments of Wiener process are independent, then
$$ \prod _{l=1}^r  \int_{t_{i_l}} ^{t_{i_l +1}}  p'_{t_{i+1}-\tau} \big( w(\tau)-w(t_i) \big)dw(\tau)   = $$
$$ =  \int_{t_{i_r}} ^{t_{i_r+1}} \cdots  \int_{t_{i_1}} ^{t_{i_1+1}} \prod _{l=1}^r  p'_{t_{i_{l}+1}-\tau_l} \big( w(\tau_l)-w(t_{i_l}) \big)dw(\tau_1)\cdots dw(\tau_{r}) \; = $$
$$  = \; \int_0^1 \cdots \int_0^1 \; \prod _{l=1}^r  p'_{t_{i_{l}+1}-\tau_l} \big( w(\tau_l)-w(t_{i_l}) \big)\; \mathbf{1}_{[ t_{i_l};t_{i_l+1} ]}(\tau_l)dw(\tau_1)\cdots dw(\tau_{r}) \; . $$
Consequently,
$$  T_k^w =  \;  \dfrac{1}{2^{\frac{k-1}{2}}\Gamma \Big( \frac{k+3}{2} \Big) }\; + \; \sum _{r=1}^{k-1} \sum_{1\leqslant i_1 <\ldots <i_r \leqslant k-1  } \int_{\Delta _k} \dfrac{1}{\prod _{j\neq i_1, \ldots ,i_r} \sqrt{ 2\pi (t_{j+1}-t_j) }} \cdot $$
$$\cdot      \;   \int_0^1 \cdots \int_0^1 \; \prod _{l=1}^r   p'_{t_{i_{l}+1}-\tau_l} \big( w(\tau_l)-w(t_{i_l}) \big)\; \mathbf{1}_{[ t_{i_l};t_{i_l+1} ]}(u_l) dw(\tau_1)\cdots dw(\tau_{r})  d\vec{t} \;=$$

$$=  \;     \dfrac{1}{2^{\frac{k-1}{2}}\Gamma \Big( \frac{k+3}{2} \Big) }\; + \;   \int_0^1  \; \sum _{r=1}^{k-1} \sum_{1\leqslant i_1 <\ldots <i_r \leqslant k-1  } \int_{\Delta _k}  \dfrac{1}{\prod _{j\neq i_1, \ldots ,i_r} \sqrt{ 2\pi (t_{j+1}-t_j) }} \cdot $$
$$\cdot    \int_0^1 \cdots \int_0^1 p'_{t_{i_{r}+1}-\tau} \big( w(\tau)-w(t_{i_r}) \big)\; \mathbf{1}_{[ t_{i_r};t_{i_r+1} ]}(\tau) \prod _{l=1}^{r-1} p'_{t_{i_{l}+1}-\tau_l} \big( w(\tau_l)-w(t_{i_l}) \big)\cdot\;
$$
$$
\cdot\mathbf{1}_{[ t_{i_l};t_{i_l+1} ]}(\tau_l)dw(\tau_1)\cdots dw(\tau_{r-1})  d\vec{t} \;dw(\tau )\; .  $$
\end{proof}
Let us establish Clark formula for a self intersection local time of Gaussian integrator. For $\vec{t}= (t_1,\ldots , t_{k} ) \in \Delta _k,\ t \in [0;1]$ denote by $ p_{\vec{t}},\ p_{\vec{t},t}$ densities of distribution of vectors
$$ X = \; \big( x(t_2)-x(t_1),\cdots, x(t_{k})-x(t_{k-1}) \big ) $$
and
$$ X(\vec{t},t) = \; \big( x(t_1+t(t_2-t_1))-x(t_1),\cdots, x(t_{k-1}+t(t_{k}-t_{k-1}))-x(t_{k-1}) \big )
$$
correspondingly. Put
$$ B_{\vec{t},t} \; = \; B \Big( A\mathbf{1}_{[t_1;t_1+t(t_2-t_1)]},\cdots ,  A\mathbf{1}_{[t_{k-1};t_{k-1}+t(t_k-t_{k-1})] }\Big ) \; ,   $$
$$ R _{\vec{t},t} \; = \; B_{\vec{t},1} - B_{\vec{t},t},$$
where as before $B(e_1,\ldots,e_n)$ is Gramian matrix constructed from elements $e_1,\ldots, e_n.$ Denote by $\Delta^{0}_{k}=\{0<t_1<\ldots<t_k<1\}.$
\begin{thm}
\label{thm6}
Suppose that for every $ 0 \; < \; t\; < \; 1,\ \vec{t}\in\Delta^{0}_{k}$ the matrix $ R _{\vec{t},t} \;$ is positive definite. Let $ p_{R_{\vec{t},t}} $ be a density of distribution $ \mathcal{N} \big( 0,R _{\vec{t},t} \big )  $. Then
$$
T_k^x \; = \; \E \; T_k^x  \; +\int^{1}_{0}\beta(\tau)dx(\tau),
$$
where
$$
\beta(\tau)= \; \int_{\Delta _k} \; \sum _{j=1}^{k-1} \mathbf{1}_{[t_j;t_{j+1}]}(\tau)   \partial _j  p_{R_{\vec{t}, \tau}} \big ( X(\vec{t},\tau) \big ) \; d\vec{t}
$$
and $ \partial _j $ denotes the j-th partial derivative.
\end{thm}
\begin{proof} Note that
\begin{align*}
\mathcal{T}\big( T_k^x \big ) (h)
& = \int_{\Delta _k} \; \mathcal{T}\Big( \prod _{i=1}^{k-1} \delta_0 (x(t_{i+1})-x(t_{i})) \Big ) (h) \; d\vec{t} \\
& = \;  \int_{\Delta _k} \E \left\lbrace \delta^{k-1} _0 \Big ( X+V \Big ) \right\rbrace \; d\vec{t},
\end{align*}
where
$ \delta^{k-1} _0(y)=\delta _0(y_1)\cdot\ldots\cdot\delta _0(y_{k-1}),\ y\in\mbR^{k-1},$
$$  V = \; \Big( \int_{t_1}^{t_2} A^{*}h(r)dr \; ,\cdots ,\; \int_{t_{k-1}}^{t_k} A^{*}h(r)dr \; \Big ) \; .$$
Hence
$$ \mathcal{T}\big( T_k^x \big ) (h) = \; \int_{\Delta _k} \;p_{\vec{t}} (V) \; d\vec{t} \; . $$
Changing variables $r=t_{j}+\theta( t_{j+1}-t_j),\ j=1,\cdots,k-1$ one can see that $V$ can be represented as follows
$$  V\; = \; \Big( \big( t_2-t_1 \big ) \int_{0}^{1} (A^{*}h ) \Big ( t_1+\theta \big( t_2-t_1 \big ) \Big ) d\theta \; ,\cdots ,\; \big( t_k-t_{k-1} \big ) \int_{0}^{1} (A^{*}h ) \Big ( t_{k-1}+\theta \big( t_k-t_{k-1} \big ) \Big ) d\theta \;\; \Big ) \; . $$
Put
$$ V(t)\; = \;  \Big( \big( t_2-t_1 \big ) \int_{0}^{t} (A^{*}h ) \Big ( t_1+\theta \big( t_2-t_1 \big ) \Big ) d\theta \; ,\cdots ,\; \big( t_k-t_{k-1} \big ) \int_{0}^{t} (A^{*}h ) \Big ( t_{k-1}+\theta \big( t_k-t_{k-1} \big ) \Big ) d\theta \;\; \Big ) \; . $$
Newton-Leibniz formula gives the relation
$$   p_{\vec{t}} (V) \; = \; p_{\vec{t}} (V(0)) \; + \; \int_0^1 \sum _{j=1}^{k-1} \partial _j p_{\vec{t}} (V(t)) \big( t_{j+1}-t_{j} \big ) (A^{*}h ) \Big ( t_{j}+ t \big( t_{j+1}-t_{j} \big ) \Big ) \; dt \;. $$
Therefore,
$$   p_{\vec{t}} (V) \; = \; p_{\vec{t}} (0) \; + \; \int_0^1 \sum _{j=1}^{k-1} \partial _j p_{\vec{t}} (V(t)) \big( t_{j+1}-t_{j} \big ) (A^{*}h ) \Big ( t_{j}+ t \big( t_{j+1}-t_{j} \big ) \Big ) \; dt \;. $$
Consequently,
$$   \mathcal{T}\big( T_k^x \big ) (h) = \E \; T_k^x  \; +  $$
\begin{equation}
\label{eq32}
+ \; \int_{\Delta _k} \int_0^1 \sum _{j=1}^{k-1} \partial _j p_{\vec{t}} (V(t)) \big( t_{j+1}-t_{j} \big ) (A^{*}h ) \Big ( t_{j}+ t \big( t_{j+1}-t_{j} \big ) \Big ) \; dt \; d\vec{t} \; .
\end{equation}
Note that
\begin{align*}
 \partial _j p_{\vec{t}} (V(t)) & =  \partial _j \big(  p_{R_{\vec{t},t}} \ast p_{\vec{t},t}  \big ) (V(t))  \\
  & =   \big( \partial _j  p_{R_{\vec{t},t}}   \big )\ast p_{\vec{t},t} (V(t)) \\
  &= \E \big( \partial _j  p_{R_{\vec{t},t}}   \big ) \Big (  X_{\vec{t},t} + V(t) \Big )  \\
& = \mathcal{T}\Big( \partial _j  p_{R_{\vec{t},t}} \big ( X_{\vec{t},t} \big )  \Big ) (h) \; .\\
\end{align*}
Changing variables $ s_j =  t_{j}+ t \big( t_{j+1}-t_{j} \big ) , \; j = 1, \ldots , k-1 \;  $ one can obtain that \eqref{eq32} equals
$$ \E \; T_k^x  \; + \; \int_{\Delta _k}\sum _{j=1}^{k-1} \int_{t_j}^{t_{j+1}}  \mathcal{T}\Big( \partial _j  p_{R_{\vec{t}, s_j}} \big ( X(\vec{t},s_j) \big )  \Big ) (h) (A^{*}h ) (s_j ) \; ds_j \; d\vec{t}.\;  $$
Therefore,
$$   \mathcal{T}\big( T_k^x \big ) (h) = \E \; T_k^x  \; + \; \mathcal{T} \left(\sum _{j=1}^{k-1}   \int_{\Delta _k} \int_{t_j}^{t_{j+1}}  \partial _j  p_{R_{\vec{t}, s_j}} \big ( X(\vec{t},s_j) \big ) \; dx(s_j) \; d\vec{t} \;  \right) (h)\; , $$
where the integral over $dx(s_j)$ is an extended stochastic integral with respect to a Gaussian integrator $x$. Consequently, Clark formula for $ T_k^x  $ has the following representation
$$T_k^x   =  \E \; T_k^x  \; + \; \int_{\Delta _k} \; \sum _{j=1}^{k-1}  \int_{t_j}^{t_{j+1}}   \partial _j  p_{R_{\vec{t}, s_j}} \big ( X(\vec{t},s_j) \big ) \;  dx(s_j) \; d\vec{t} \;=$$
$$
=\E \; T_k^x  \; + \int^1_0\; \; \int_{\Delta _k} \; \sum _{j=1}^{k-1} \mathbf{1}_{[t_j;t_{j+1}]}(\tau)   \partial _j  p_{R_{\vec{t}, \tau}} \big ( X(\vec{t},\tau) \big ) \; d\vec{t}\;dx(\tau).
$$
\end{proof}

\end{document}